\documentclass[12pt,a4paper]{amsart}
\usepackage{amsfonts}
\usepackage[top=35mm, bottom=35mm, left=30mm, right=30mm]{geometry}
\usepackage{mathptmx}
\usepackage{hyperref}

\newtheorem{thm}{Theorem}[section]
\newtheorem{lem}[thm]{Lemma}

\newtheorem{prop}[thm]{Proposition}

\newtheorem{ex}[thm]{Example}
\theoremstyle{definition}
\newtheorem{de}[thm]{Definition}
\newtheorem{rem}[thm]{Remark}
\newtheorem*{rem*}{Remark}
\newtheorem{ques}{Question}
\numberwithin{equation}{section}
\begin{document}
\title{Point transitivity, $\Delta$-transitivity and multi-minimality}

\author[Z.~Chen]{Zhijing Chen}
\address[Z.~Chen]{Department of Mathematics, Sun Yat-Sen University, Guangzhou 510275, P. R. China}
\email{chzhjing@mail2.sysu.edu.cn}
\author[J.~Li]{Jian Li}
\address[J.~Li]{Department of Mathematics, Shantou University, Shantou, Guangdong 515063, P.R. China}
\thanks{Corresponding author: Jian Li (lijian09@mail.ustc.edu.cn)}
\email{lijian09@mail.ustc.edu.cn}
\author[J.~L\"u]{Jie L\"U}
\address[J.~L\"u]{School of Mathematics, South China Normal University, Guangzhou 510631, P. R. China}
\email{ljie@scnu.edu.cn}

\date{\today}

\begin{abstract}
Let $(X, f)$ be a topological dynamical system and $\mathcal{F}$ be a Furstenberg family
(a collection of subsets of $\mathbb{N}$ with hereditary upward property).
A point $x\in X$ is called an $\mathcal{F}$-transitive point if
for every non-empty open subset $U$ of $X$ the entering time set of $x$ into $U$,
$\{n\in\mathbb{N}: f^{n}(x)\in U\}$, is in $\mathcal{F}$;
the system $(X,f)$ is called $\mathcal {F}$-point transitive if there exists some $\mathcal{F}$-transitive point.
In this paper, we first discuss the connection between $\mathcal{F}$-point transitivity and $\mathcal{F}$-transitivity,
and show that weakly mixing and strongly mixing systems can be characterized by $\mathcal{F}$-point transitivity,
completing results in [Transitive points via Furstenberg family, Topology Appl. 158 (2011), 2221--2231].
We also show that multi-transitivity,
$\Delta$-transitivity and multi-minimality can also be characterized by $\mathcal{F}$-point transitivity,
answering two questions proposed by  Kwietniak and Oprocha
[On weak mixing, minimality and weak disjointness of all iterates,  Erg. Th. Dynam. Syst., 32 (2012), 1661--1672].
\end{abstract}

\maketitle


\section{Introduction}

Throughout this paper a \emph{topological dynamical system} (or just \emph{dynamical system}, \emph{system})
is a pair $(X, f)$, where $X$ is a compact metric space
and $f:X\rightarrow X$ is a continuous map.
A topological dynamical system $(X,f)$ is called \emph{transitive}
if for every two non-empty open subsets $U$ and $V$ of $X$
there exists a positive integer $k$ such that $U\cap f^{-k}(V)$ is not empty.

It is well known that the study of transitive systems and its classification plays a big role in topological dynamics.
There are several ways to classify transitive systems.
One of them started  by  Furstenberg is to classify transitive systems
by the hitting time sets of two non-empty open subsets.
Let $\mathcal {F}$ be a Furstenberg family (a collection of subsets of $\mathbb{N}$ with hereditary upward property).
We call $(X, f)$ is \emph{$\mathcal {F}$-transitive} if for every two non-empty open subsets $U, V$ of $X$
the hitting time set of $U$ and $V$, $N(U,V):=\{n\in\mathbb{N}: U\cap f^{-n}(V)\neq\emptyset\}$, is in $\mathcal {F}$.

We say that $(X,f)$ is \emph{weakly mixing} if the product system $(X\times X,f\times f)$ is transitive.
In his seminal paper \cite{Furstenberg-1967},
Furstenberg showed that a topological dynamical system $(X,f)$ is weakly mixing if and only if it is \{thick
sets\}-transitive.
The authors in~\cite{A97}, \cite{G04}, \cite{HY04} and~\cite{W-Huang-X-Ye-2005}
have successfully classified many transitive systems by using this way.
However, proper families have not been found for some important classes of transitive systems
such as M-systems and E-systems.
Recently, the second author of this paper proposed a new way in~\cite{L2011},
called $\mathcal {F}$-point transitivity, to classify these systems.
Let $\mathcal{F}$ be a Furstenberg family. A point $x\in X$ is called an \emph{$\mathcal {F}$-transitive point} if
for every non-empty open subset $U$ of $X$ the entering time set of $x$ into $U$,
$N(x,U):=\{n\in \mathbb{N}: f^{n}(x) \in U\}$,
is in $\mathcal {F}$;
the system $(X,f)$ is called \emph{$\mathcal {F}$-point transitive} if there exists some $\mathcal {F}$-transitive point.
It is shown in~\cite{L2011} that E-systems, M-systems, weakly mixing E-systems, weakly mixing M-systems
and HY-systems can be characterized by $\mathcal{F}$-point transitivity.
But the following problem is still open.

\begin{ques}~\label{ques:wm}
Can weakly mixing systems be characterized by $\mathcal{F}$-point transitivity?
\end{ques}

In section 3, we first discuss  the connection between $\mathcal {F}$-point transitivity
and $\mathcal {F}$-transitivity, and show that weakly mixing systems and strongly mixing systems
can be also characterized by $\mathcal{F}$-point transitivity, giving a positive answer to Question~\ref{ques:wm}.

In fact, our characterization of weak mixing also answers the Problem~1 in~\cite{Erdmann-Peris-10}
in the framework of hypercyclic operators on a Fr\'echet space.
A \emph{linear dynamical systems} is a pair $(X,T)$, where $X$ is a Fr\'echet space
(or more general a topological vector space) and $T:X\to X$ is a continuous and linear operator.
We refer the reader to~\cite{Erdmann-Peris-11} for the details
concerning linear dynamical systems.
In~\cite{Bayart-Grivaux-06}, Bayart and Grivaux introduced the notion of frequently hypercyclic
operators, in other words, $\{$positive lower density sets$\}$-point transitivity in our setting.
It is shown in~\cite{Erdmann-Peris-05} that frequently hypercyclic operators are weakly mixing.
Latter, several authors show that $\{$positive upper Banach density sets$\}$-point
transitive operators are weakly mixing (see~\cite{Moothathu-2009} or~\cite{Peris-05}).
The authors in~\cite{Bayart-09} studied in detail how fast the integers of the sets $N(x,U)$ could
increase to ensure that the operator is weakly mixing.
In~\cite{Erdmann-Peris-10}, Grosse-Erdmann and Peris asked that
within the framework of hypercyclic operators on a Fr\'echet space, is there a `nice' condition
expressed in terms of the sets $N(x,U)$ that characterizes the weakly mixing property?
Following Birkhoff transitivity theorem (see ~\cite[Theorem 2.19]{Erdmann-Peris-11}),
our characterization of weak mixing can be applied to characterize weak mixing operators.

In~\cite{FW78}, Furstenberg and Weiss showed that the famous
topological multiple recurrence theorem as following,
which can be applied to prove the van der Waerden theorem in combinatorics.
\begin{thm}[\cite{FW78,F81}]
Let $(X,f)$ be a minimal dynamical system.
Then for each $m\in \mathbb{N}$, there exists a residual subset $Y$ of $X$ such that
for every point $x\in Y$ the diagonal $m$-tuple $(x,x,\dotsc, x)$ is a recurrent point in $X^m$
under the action of $f\times f^2\times\dotsb\times f^m$,
that is, there exists a sequence $\{n_k\}$ of positive integers such that
\[\lim_{k\to\infty} f^{in_k}(x)=x, \text{ for } i =1,2,\dotsc,m.\]
\end{thm}

This result highlights the importance of researching the properties of $f\times f^2\times\dotsb\times f^n$.
Using the structure theory of minimal systems, Glasner proved the following interesting result.

\begin{thm}[\cite{G94,E-Glasner-2000}]\label{thm:G2000}
Let $(X,f)$ be a weakly mixing minimal system.
Then for each $m\in \mathbb{N}$, there exists a residual subset $Y$ of $X$ such that
for every point $x\in Y$  the diagonal $m$-tuple $(x,x,\dotsc, x)$ has a dense orbit in $X^m$
under the action of $f\times f^2\times\dotsb\times f^m$, that is
\[\{(f^n(x), f^{2n}(x),\dotsc,f^{mn}(x)):\ n=0,1,\dotsc\}\text{ is dense in }X^m.\]
\end{thm}

In 2010, Moothathu~\cite{Moothathu-2010} provided a simplified proof of Theorem~\ref{thm:G2000}
without resorting to the heavy machinery of the structure theory of minimal systems.
He also asked whether $(X,f)$ should possess the following two properties:
\begin{enumerate}
  \item For each $m\in\mathbb{N}$, $(X^m,f\times f^2\times \dotsb\times f^m)$ is transitive.
  \item For each $m\in\mathbb{N}$, there exists a residual subset $Y$ of $X$ such that
for every point $x\in Y$  the diagonal $m$-tuple $(x,x,\dotsc, x)$ has a dense orbit in $X^m$
under the action of $f\times f^2\times\dotsb\times f^m$.
\end{enumerate}
Following~\cite{Moothathu-2010}, we will say that $(X,f)$ is \emph{multi-transitive} if it satisfies (1) and
that $(X,f)$ is \emph{$\Delta$-transitive} if it satisfies (2).
It is shown in~\cite{Moothathu-2010} that weak mixing, multi-transitivity and $\Delta$-transitivity are equivalent
for minimal  systems.

Moothathu asked whether there are implications between multi-transitivity and weak mixing for general (non-minimal) systems.
In 2012, Kwietniak and Oprocha~\cite{D-Kwietniak-P-Oprocha-2010} showed that in general there
is no connection between weak mixing and multi-transitivity by constructing examples of
weakly mixing but non-multi-transitive and multi-transitive but non-weakly mixing systems.
They proposed the following natural problem.

\begin{ques}
Is there any non-trivial characterization of multi-transitive weakly mixing systems?
\end{ques}

In a recent paper~\cite{Chzhj-Ljian-Lj},
we show that multi-transitivity can be characterized by $\mathcal{F}$-transitivity.
In section 4, we will show that multi-transitivity can be also characterized by $\mathcal{F}$-point transitivity.

It is shown in~\cite{Moothathu-2010} that $\Delta$-transitivity implies weak mixing,
but there exists a strongly mixing system which is not $\Delta$-transitive.
Another natural problem is the following:
\begin{ques}
Is there any non-trivial characterization of $\Delta$-transitive systems?
\end{ques}
In section 5, we will show that $\Delta$-transitive systems can be characterized
$\mathcal{F}$-point transitivity, while they can not be
characterized by $\mathcal{F}$-transitivity.

A dynamical system is called \emph{multi-minimal} if  for every $n\in\mathbb{N}$,
the system $(X^{n}, f\times f^{2}\times \dotsb\times f^{n})$ is minimal.
In~\cite{Moothathu-2010}, Moothathu was not aware that there have been some study in this topic.
But the terminology is slightly different to multi-minimality.
See a brief introduction in~\cite{D-Kwietniak-P-Oprocha-2010}.
Kwietniak and Oprocha~\cite{D-Kwietniak-P-Oprocha-2010} also
remarked that although every weakly mixing minimal system is multi-transitive,
it is not necessarily  multi-minimal.
The discrete horocycle flow $h$ is an example of a weakly mixing minimal homeomorphism
but not multi-minimal (see~\cite[pp.26,~105--110]{E-Glasner-2003}).
By this observation, Kwietniak and Oprocha proposed the following problem in~\cite{D-Kwietniak-P-Oprocha-2010}.

\begin{ques}
Is there any non-trivial characterization of multi-minimality in terms of some
dynamical properties?
\end{ques}
In section 6, we answer this question by showing that multi-minimal systems can be also characterized
by $\mathcal{F}$-point transitivity, while they can not be characterized by $\mathcal{F}$-transitivity.

\section{Preliminaries}
In this paper, the sets of integers, non-negative integers and positive integers
are denoted by $\mathbb{Z}$, $\mathbb{Z}_+$ and $\mathbb{N}$ respectively.
For $r\in\mathbb{N}$, denote $\mathbb{N}^r=\mathbb{N}\times \mathbb{N}\times \dotsb\times \mathbb{N}$ ($r$-copies) and
$\mathbb{N}^r_*=\{(n_1,n_2,\dotsc,n_r)\in\mathbb{N}^r:\ n_1<n_2<\dotsb<n_r\}$.
For $M$, $N\in\mathbb{N}$ with $M\leq N$, denote by $[M, N]=\{M,M+1,\ldots,N\}$ and $[M,+\infty)=\{M,M+1,\ldots\}$.

\subsection{Topological dynamics}
Let $(X, f)$ be a dynamical system.
For two subsets $U$, $V$ of $X$, we define the \emph{hitting time set of $U$ and $V$} by
\[N(U,V)=\{n\in\mathbb{N}:\ f^n(U)\cap V\neq\emptyset\}=\{n\in\mathbb{N}:\ U\cap f^{-n}(V)\neq\emptyset\}.\]
For a point $x\in X$ and a subset $U$ of $X$,
we define the \emph{entering time set of $x$ into $U$} by
\[N(x,U)=\{n\in\mathbb{N}:\ f^n(x)\in U\}.\]
When there is more than one action on the underlying space,
we will use the notations $N_f(U,V)$ and $N_f(x,U)$ to avoid ambiguity.

The system $(X,f)$ is called \emph{minimal} if it has no proper closed invariant subsets,
that is, if $K\subset X$ is non-empty, closed and $f(K)\subset K$, then $K=X$.
A point $x\in X$ is called \emph{minimal} if it is contained in some
minimal subsystem of $(X,f)$.

We say that $(X,f)$ is \emph{(topologically) transitive} if for every
two non-empty open subsets $U$ and $V$ of $X$, the hitting time set $N(U,V)$ is non-empty;
\emph{totally transitive} if $(X,f^n)$ is transitive for any $n\in\mathbb{N}$;
\emph{weakly mixing} if the product system $(X\times X,f\times f)$ is transitive;
\emph{strongly mixing} if for every
two non-empty open subsets $U$ and $V$ of $X$, the hitting time set $N(U,V)$ is cofinite,
that is, there exists $N\in\mathbb{N}$ such that $[N,+\infty)\subset N(U,V)$.

For a point $x\in X$, denote the \emph{orbit} of $x$ by  $Orb(x,f)=\{x,f(x),f^2(x),\dotsc,f^n(x),\dotsc\}$.
Let $\omega(x,f)$ be the \emph{$\omega$-limit set} of $x$, i.e., $\omega(x,f)$ is the limit set of $Orb(x,f)$.
A point $x\in X$ is called a \emph{recurrent point} if $x\in\omega(x,f)$, and a \emph{transitive point}
if $\omega(x,f)=X$. For a transitive system $(X,f)$, a point $x\in X$ is a transitive point
if and only if the orbit of $x$ is dense in $X$.
It is not hard to verified that a system $(X,f)$ is transitive if and only if
the collection of all transitive points, denoted by $Trans(X,f)$, is a dense $G_\delta$ subset of $X$ and
$(X,f)$ is minimal if and only if $Trans(X,f)=X$.

A dynamical system $(X, f)$ is an \emph{E-system} if it is transitive and there is an invariant Borel
probability measure $\mu$ with full support, i.e.,
$supp(\mu)=\{x\in X:$  for every open neighborhood  $U$ of $x$, $\mu(U)>0\}=X$;
an \emph{M-system} if it  is transitive and has dense minimal points.

Let $(X, f)$ be a dynamical system. We say that $(X, f)$ has
\emph{dense small periodic sets} if for every non-empty open subset $U$ of $X$
there exists a closed subset $Y$ of $U$ and $k \in\mathbb{N}$
such that $Y$ is invariant for $f^ k$ (i.e., $f^k(Y)\subset Y$).
Clearly, if $(X,f)$ has dense periodic points,
then it  has dense small periodic sets.
We say that $(X,f)$ is an \emph{HY-system}
if it is totally transitive and has dense small periodic sets \cite{L2011}.

Let $\Lambda = \{0, 1\}$ be equipped with the discrete topology.
Let $\Sigma =\Lambda^{\mathbb{Z}_+}$ denote the set of all infinite sequence of symbols in $\Lambda$
indexed by the non-negative integers $\mathbb{Z}_+$ with the product topology.
The \emph{shift} transformation is a continuous map $\sigma:\Sigma\to\Sigma$ given by
$(\sigma(x))_i=x_{i+1}$, where $x=x_0x_1x_2\dotsb$.
A \emph{subshift} is any non-empty closed subset $X$ of $\Sigma$ such that $\sigma(X)\subset X$.

A \emph{word} of length $k$ is a finite sequence $\omega=\omega_1\omega_2\dotsb\omega_k$
of elements of $\{0,1\}$.
The length of a word $\omega$ is denoted as $|\omega|$.
We say that a word $\omega=\omega_1\omega_2\dotsb\omega_k$
appears in $x=(x_i)\in\Sigma$ at position $t$
if $x_{t+j-1}=\omega_j$ for $j=1,2,\dotsc,k$.
A cylinder in a subshift $X\subset \Sigma$ is any set $[u]=\{x\in X: x_0x_1\dotsc x_{n-1}=u\}$, where $u$ is a word with length $n$.
The family of cylinders in a subshift $X\subset \Sigma$ is a base of the topology of $X$ inherited from $\Sigma$.

For any $P\subset\mathbb{N}$ we define
$$\Sigma_P=\{x\in\Sigma:x_i=x_j=1\Rightarrow |i-j|\in P\cup\{0\}\}.$$
It is easily verified that $\Sigma_P$ is a subshift. We will call a subshifts defined in this
way the \emph{spacing subshifts} (see \cite{K-Lau-A-Zame-1973} and \cite{Banks-Nguyen-Oprocha-etc-2010}).
For a cylinder $[w]$  in $\Sigma$, let $[w]_P=[w]\cap \Sigma_P$.

\subsection{Furstenberg families}
Let $\mathcal{P}$ denote the collection of all subsets of $\mathbb{N}$.
A subset $\mathcal F$ of $\mathcal P$ is called a \emph{Furstenberg family} (or just a \emph{family}),
if it is hereditary upward, i.e.,
\[\text{$F_1 \subset F_2$ and $F_1\in\mathcal F$ imply
$F_2\in\mathcal{F}$.} \]
 A family $\mathcal{F}$ is called \emph{proper} if it is a non-empty
proper subset of $\mathcal P$,
i.e., neither empty nor all of $\mathcal P$.
Any non-empty collection $\mathcal A$ of subsets of $\mathbb{N}$ naturally generates a
family
\[[\mathcal A] = \{F\subset\mathbb{N}: A \subset F \text{ for some } A\in\mathcal{A}\}.\]
For a family $\mathcal F$, the {\em dual family} of $\mathcal F$, denoted by $\kappa\mathcal F$, is
\[\{F\in \mathcal P: F\cap F'\neq \emptyset \text{ for every }  F'\in \mathcal F\}.\]
Sometimes the dual family $\kappa\mathcal F$ is also denoted by $\mathcal F^*$.
Let $\mathcal F_{inf}$ be the family of all infinite subsets of $\mathbb{N}$.
It is easy to see that its dual family $\kappa\mathcal F_{inf}$
is the family of all cofinite subsets of $\mathbb{N}$, denoted by $\mathcal F_{cf}$.

A subset $F$ of $\mathbb{N}$ is called \emph{thick}
if it contains arbitrarily long runs of positive integers,
i.e., for every  $n\in\mathbb{N}$  there exists some  $a_n\in\mathbb{N}$
 such that $[a_n,a_n+n]\subset F$,
and \emph{syndetic} if there is $N\in\mathbb{N}$ such that $[n,n+N]\cap F\neq \emptyset$ for every $n\in\mathbb{N}$.
The families of all thick sets and syndetic sets are denoted by $\mathcal F_{t}$ and $\mathcal F_{s}$ respectively.
It is easy to see that $\kappa \mathcal F_s=\mathcal F_t$.

For $n\in\mathbb{Z}$ and $F\subset \mathbb{N}$, put
\[F+n=\{k+n:\ k\in F\}\cap \mathbb{N}.\]
A family $\mathcal {F}$ is called \emph{translation $\pm$ invariant} if for every $n\in \mathbb{Z}_+$
and every $F\in\mathcal {F}$, we have $F\pm n\in \mathcal {F}$.
It is easy to see that $\mathcal {F}_{inf}$,
$\mathcal{F}_{cf}$, $\mathcal{F}_t$ and $\mathcal{F}_s$
are translation $\pm$  invariant.

For a subset $F\subset\mathbb{N}$, the \emph{difference set} of $F$ is $F-F=\{a-b:\ a,b\in F\text{ and }a>b\}$.
For a family $\mathcal{F}$, the \emph{difference family} of $\mathcal{F}$ is
$\Delta(\mathcal{F})=[\{F-F:\ F\in\mathcal{F}\}]$.

\subsection{Topological dynamics via Furstenberg families}
The idea of using Furstenberg families to
describe dynamical properties goes back at least to Gottschalk and Hedlund~\cite{GH55}.
It was developed further by Furstenberg~\cite{F81}.
For a systematic study and recent results, see~\cite{A97}, \cite{G04}, \cite{HY04} and~\cite{W-Huang-X-Ye-2005}.

Let $(X,f)$ be a dynamical system and $\mathcal{F}$ be a family.
A point $x\in X$ is called \emph{$\mathcal{F}$-recurrent}
if for every neighborhood $U$ of $x$ the entering time set $N(x, U)$ is in $\mathcal{F}$.
In~\cite{GH55}, Gottschalk and Hedlund characterized the entering time sets of minimal points.

\begin{lem}[\cite{GH55}] Let $(X, f)$ be a dynamical system and $x\in X$.
Then $x$ is a minimal point if and only if it is $\mathcal F_{s}$-recurrent.
\end{lem}

Recall that a dynamical system $(X,f)$ is called {\em $\mathcal F$-transitive}
if for every two non-empty open subsets $U,V$ of $X$ the hitting time set $N(U,V)$ is in $\mathcal F$;
\emph{$\mathcal{F}$-central} if for every non-empty open subset $U$ of $X$ the hitting time set
$N(U,U)$ is in $\mathcal{F}$;
{\em $\mathcal F$-mixing} if $(X\times X, f\times f)$ is $\mathcal F$-transitive.

\begin{lem}[\cite{Furstenberg-1967,A97}] \label{lem:weak-mixing}
Let $(X, f)$ be a dynamical system and $\mathcal F$ be a family. Then
\begin{enumerate}
\item $(X, f)$ is weakly mixing if and only if it is $\mathcal F_{t}$-transitive.
\item $(X, f)$ is strongly mixing if and only if it is $\mathcal F_{cf}$-transitive.
\item $(X,f)$ is $\mathcal F$-mixing if and only if it is $\mathcal F$-transitive and weakly mixing.
\end{enumerate}
\end{lem}
The following lemma describes the relationship between $\mathcal{F}$-transitivity and $\mathcal{F}$-center.
\begin{lem}[\cite{A97}]\label{lem:F-trans-equ-F-center-trans}
Let $(X,f)$ be a dynamical system and let $\mathcal{F}\subset \mathcal {F}_{inf}$
be a proper translation $+$ invariant family.
Then the system $(X,f)$ is $\mathcal{F}$-transitive if and only if it is transitive and $\mathcal{F}$-central.
\end{lem}

\begin{de}[\cite{L2011}]
Let $(X,f)$ be a dynamical system and $\mathcal F$ be a family.
A point $x \in X$ is called an \emph{$\mathcal{F}$-transitive point} if for every non-empty open subset $U$ of $X$,
the entering time set $N(x, U)$ is in $\mathcal{F}$.
Denote the set of all $\mathcal{F}$-transitive points by $Trans_{\mathcal F}(X, f)$.
We say that  the system $(X, f)$ is \emph{$\mathcal{F}$-point transitive}
if there exists some $\mathcal{F}$-transitive point in $X$.
\end{de}

The following lemma is easy to verified.
\begin{lem}\label{lem:minimal-Fs}
Let $(X,f)$ be a dynamical system. Then
\begin{enumerate}
\item $(X,f)$ is transitive if and only if it is $\mathcal{F}_{inf}$-point transitive.
\item $(X,f)$ is minimal if and only if it is $\mathcal{F}_s$-point transitive if and only if
$Trans_{\mathcal{F}_s}(X,f)=X$.
\end{enumerate}
\end{lem}

We collect some results of $\mathcal{F}$-point transitive as following and
refer the reader to~\cite{L2011} for the definitions of Furstenberg families.

\begin{thm}[\cite{HKY07,W-Huang-X-Ye-2005}] \label{thm:E-M-system}
Let $(X,f)$ be a dynamical system. Then
\begin{enumerate}
\item $(X,f)$ is an E-system if and only if it is $\{\text{positive upper Banach density sets}\}$-point transitive.
\item $(X,f)$ is an M-system if and only if it is $\{\text{piecewise syndetic sets}\}$-point transitive.
\end{enumerate}
\end{thm}
\begin{thm}[\cite{L2011}] \label{thm:WME-M-system}
Let $(X,f)$ be a dynamical system. Then
\begin{enumerate}
\item $(X,f)$ is a weakly mixing E-system if and only if it is  $\{\text{D-sets}\}$-point transitive.
\item $(X,f)$ is a weakly mixing M-system if and only if it is $\{\text{central sets}\}$-point transitive.
\item $(X,f)$ is an HY-system if and only if it is $\{\text{weakly thick sets}\}$-point transitive.
\end{enumerate}
\end{thm}

\section{The connection between \texorpdfstring{$\mathcal F$}{F}-transitivity
and \texorpdfstring{$\mathcal F$}{F}-point transitivity}
In this section, we discuss the connection between $\mathcal F$-transitivity and $\mathcal F$-point transitivity.
First, we show that many classes of transitive systems can not be classified by $\mathcal{F}$-transitivity.

\begin{prop}\label{pro:Not-F-trans-iter}
Let $\mathfrak{A}$ be a class of transitive systems.
If  $\mathfrak{A}$ contains at least one non-periodic system and
there exists some strongly mixing system which is not in $\mathfrak{A}$,
then there does not exist a Furstenberg family $\mathcal{F}$ such that
a dynamical system is in $\mathfrak{A}$ if and only if it is $\mathcal{F}$-transitive.
\end{prop}

\begin{proof}
Suppose that there exists a Furstenberg family $\mathcal{F}$ such that a dynamical system is in $\mathfrak{A}$
if and only if it is $\mathcal{F}$-transitive.
First, we  show that $\mathcal{F}_{cf}\subset \mathcal{F}$.
Let $(X,f)$ be a non-periodic system in $\mathfrak{A}$ and $F\in\mathcal{F}_{cf}$.
There exists $N\in\mathbb{N}$ such that $[N,\infty)\subset F$.
Pick a transitive point $x$ in $(X,f)$. Since $(X,f)$ is non-periodic, $f^i(x)\neq f^j(x)$
for all positive integers $i,j$ with $i\neq j$.
By the continuity of $f$, there is an open neighborhood $U$ of $x$
such that $U\cap f^i(U)=\emptyset $ for $i=1,2,\dotsc,N$.
Then $N(U,U)\subset[N,\infty)$.
By $\mathcal{F}$-transitivity of $(X,f)$, we have $N(U,U)\in\mathcal{F}$.
Then $[N,\infty)\in\mathcal{F}$ and $F\in\mathcal{F}$.
This implies that  $\mathcal{F}_{cf}\subset\mathcal{F}$.

Let $(Y,g)$ be a strongly mixing system which is not in $\mathfrak{A}$.
By Lemma~\ref{lem:weak-mixing}, $(Y,g)$ is $\mathcal{F}_{cf}$-transitive and then $\mathcal{F}$-transitive.
This is a contradiction.
\end{proof}

\begin{rem}
By Lemma~\ref{lem:minimal-Fs}, Theorems~\ref{thm:E-M-system} and~\ref{thm:WME-M-system},
the collections of all minimal systems, E-systems, M-systems, weakly mixing E-systems,
weakly mixing M-systems and HY-systems can be classified by $\mathcal{F}$-point transitivity.
But there exists a strongly mixing system which is not an E-system (see~\cite{HZ02}).
Then by Proposition~\ref{pro:Not-F-trans-iter},
all of those collections can not be classified by $\mathcal{F}$-transitivity.
\end{rem}

We shall use the following lemma which is a folklore result, for completeness we provide a proof.

\begin{lem}\label{lem:NUV}
Let $(X,f)$ be a transitive system and $x\in X$ be a transitive point.
Then for every two non-empty open subsets $U$ and $V$ of $X$, we have
\[N(U,V)=N(x,V)-N(x,U).\]
\end{lem}
\begin{proof}
Let $n\in N(U,V)$. Then $U\cap f^{-n}(V)$ is a non-empty open subset of $X$.
Since $x$ is transitive point,
there exists $k\in \mathbb{N}$ such that
$f^k(x)\in U\cap f^{-n}(V)$.
Then $k\in N(x,U)$ and $k+n\in N(x,V)$, which imply $n\in N(x,V)-N(x,U)$.

Now assume that $n\in N(x,V)-N(x,U)$.
Then there exist $n_1\in N(x,U)$ and $n_2\in N(x,V)$ with $n_2-n_1=n$.
That is $f^{n_1}(x) \in U$ and $f^{n_2}(x)\in V$.
Let $y=f^{n_1}(x)\in U$.
We have $f^n(y)=f^{n}(f^{n_1}(x))=f^{n_2}(x)\in V$.
This shows that $n\in N(U,V)$.
\end{proof}

Recall that for a family  $\mathcal{F}$, the \emph{difference family} of $\mathcal F$ is
$\Delta(\mathcal{F}) = [\{F-F:\ F\in\mathcal{F}\}]$.
By Lemma~\ref{lem:NUV}, we have the following easy fact.

\begin{lem}\label{lem:F-center-F-point-trans}
Let $(X,f)$ be a dynamical system and $\mathcal{F}$ be a family.
If $(X,f)$ is $\mathcal{F}$-point transitive, then it is $\Delta(\mathcal{F})$-central.
\end{lem}

To show that $\mathcal{F}$-transitivity implies what is kind of point transitivity,
we should introduce a new kind of family.
For a family $\mathcal{F}$, the \emph{reverse difference family} of $\mathcal F$ is defined as
\[\nabla(\mathcal{F}) = \{F \subset\mathbb{N}:\ F -F\in\mathcal{F}\}.\]

\begin{prop}\label{prop:tans-F-center-equi-nablaF-trans}
Let $(X,f)$ be a dynamical system and $\mathcal{F}$ be a family.
Then the following conditions are equivalent.
\begin{enumerate}
  \item $(X,f)$ is transitive and $\mathcal{F}$-central;
  \item $(X,f)$ is $\nabla(\mathcal{F})$-point transitive;
  \item $Trans_{\nabla(\mathcal{F})}(X,f)=Trans(X,f)\neq\emptyset$.
\end{enumerate}
\end{prop}
\begin{proof}
(1)$\Rightarrow$(3) Let $x$ be a transitive point of $(X,f)$.
We want to show that $x$ is also a $\nabla(\mathcal{F})$-transitive point.
Let $U$ be a non-empty open subset of $X$.
By Lemma~\ref{lem:NUV}, we have $N(U,U)=N(x,U)-N(x,U)$.
Since $(X,f)$  is $\mathcal{F}$-central, $N(U,U)\in\mathcal{F}$, then $N(x,U)\in\nabla(\mathcal{F})$.
Hence $x$ is a $\nabla(\mathcal{F})$-transitive point.

(3)$\Rightarrow$(2) is obvious.

(2)$\Rightarrow$(1) Clearly $(X,f)$ is transitive. Now pick a $\nabla(\mathcal{F})$-transitive point $x\in X$.
Then for every non-empty open subset $U$ of $X$, one has $N(U,U)=N(x,U)-N(x,U)$.
Since $N(x,U)$ is in $\nabla(\mathcal{F})$,
by the definition of the reverse difference family, we have $N(U,U)$ is in $\mathcal{F}$.
Then $(X,f)$ is $\mathcal{F}$-central.
\end{proof}

\begin{prop}\label{lem:F-trans-equi-nablaF-trans}
Let $(X,f)$ be a dynamical system and $\mathcal{F}\subset \mathcal {F}_{inf}$ be a proper translation $+$ invariant family.
Then $(X,f)$ is $\mathcal{F}$-transitive  if and only if $(X,f)$ is $\nabla(\mathcal{F})$-point transitive.
\end{prop}
\begin{proof}
It follows from Lemma~\ref{lem:F-trans-equ-F-center-trans} and  Proposition~\ref{prop:tans-F-center-equi-nablaF-trans}.
\end{proof}

It is not known in~\cite{L2011} that whether we can classify weak mixing by $\mathcal{F}$-point transitivity.
Now by Lemma~\ref{lem:weak-mixing}, Proposition~\ref{lem:F-trans-equi-nablaF-trans} and the fact that
both $\mathcal {F}_t$ and $\mathcal {F}_{cf}$ are proper translation $+$ invariant families,
we immediately obtain the following result.
\begin{thm}\label{thm:weak-str-mix-nablaF-poin-trans}
Let $(X,f)$ be a dynamical system.
Then
\begin{enumerate}
  \item $(X,f)$ is weakly mixing if and only if it is $\nabla(\mathcal{F}_t)$-point transitive;
  \item $(X,f)$ is strongly mixing if and only if it is $\nabla(\mathcal{F}_{cf})$-point transitive.
\end{enumerate}
\end{thm}

\begin{rem}
Note that Theorem~\ref{thm:weak-str-mix-nablaF-poin-trans} also holds when $X$ is a Polish space.
Thus, by Birkhoff transitivity theorem (that is, a linear operator $f$ on a separable Fr\'{e}chet space is hypercyclic
if and only if it is topologically transitive, see~\cite[Theorem~2.19]{Erdmann-Peris-11}),
Theorem~\ref{thm:weak-str-mix-nablaF-poin-trans} also holds for linear dynamical systems,
answering the Problem~1 in~\cite{Erdmann-Peris-10} in the framework of hypercyclic operators on a Fr\'echet space.
\end{rem}

\section{Multi-transitivity}
It is shown in~\cite{Chzhj-Ljian-Lj} that we can characterize multi-transitivity by $\mathcal {F}$-transitivity.
In this section, we show  that  multi-transitivity can be also characterized by $\mathcal {F}$-point transitivity.
\begin{de}\label{def:muti-trans}
Let $(X, f)$ be a dynamical system and $\mathbf{a}=(a_{1}, a_{2}, \dotsc, a_r)$ be a vector in $\mathbb{N}^r$.
We say that $(X,f)$ is
\begin{enumerate}
\item \emph{multi-transitive with respect to the vector $\mathbf{a}$}
(or briefly \emph{$\mathbf{a}$-transitive}),
if the product system $(X^r,  f^{(\mathbf{a})})$ is transitive
where $f^{(\mathbf{a})}=f^{a_{1}}\times f^{a_{2}}\times \dotsb \times f^{a_{r}}$;
\item \emph{multi-transitive}
if it is multi-transitive with respect to $(1,2,\dotsc,n)$ for any $n\in\mathbb N$;
\item \emph{strongly multi-transitive}
if it is multi-transitive with respect to any vector in $\mathbb{N}^r$ and any $r\in\mathbb{N}$.
\end{enumerate}
\end{de}

\begin{rem}\label{lem:mul-tran-tran-(2,3)-thick}
The authors in~\cite{D-Kwietniak-P-Oprocha-2010} showed that there is no implication between weak mixing and
multi-transitivity by constructing two special spacing subshifts, one is a multi-transitive non-weakly mixing
system, and the other is a weakly mixing non-multi-transitive system. In fact, for every $m > 2$ they
constructed a weakly mixing spacing subshift which is multi-transitive with respect to $(1, 2, \ldots,m-1)$ but
not for $(1, 2, \ldots, m)$.

Here, we provide another example which is similar to examples in~\cite{D-Kwietniak-P-Oprocha-2010} and
show that in general $(2,3)$-transitivity could not imply $(1,2)$-transitivity.

\begin{ex}
Put $P=\bigcup_{k=1}^{\infty}\{2^{2k-1},2^{2k-1}+1,\cdots,2^{2k}-1\}$.
Then the spacing subshift $(\Sigma_P, \sigma_P)$ is $(2,3)$-transitive  but  not $(1,2)$-transitive.
\end{ex}

\begin{proof}
Fix open cylinders $[u^{(1)}]_P$, $[u^{(2)}]_P$, $[v^{(1)}]_P$ and $[v^{(2)}]_P$ of $\Sigma_P$.
Without loss of generality, we assume that there is $k\geq 2$ such that
$t:=\vert u^{(i)}\vert =\vert  v^{(i)}\vert =2^{2k-2}$ for any $i=1,2$, where $\vert  u\vert $ denote the length of $u$.
Set $s=2^{2k}+2^{2k-3}$ and define
$$w^{(i)}=u^{(i)}0^{(i+1)s-t}v^{(i)}  \text{ for } i=1,2.$$
Then
$$[w^{(i)}]_P\subset (\sigma_P^{i+1})^{-s}[v^{(i)}]_P\cap [u^{(i)}]_P,$$
which implies that
$$[w^{(1)}]_P\times [w^{(2)}]_P\subset (\sigma^{2}\times\sigma^{3})^{-s}([u^{(1)}]_P\times [u^{(2)}]_P)\cap ([v^{(1)}]_P\times [v^{(2)}]_P).$$
Follows from definition of $w^{(i)}$,
we have $$Sp(w^{(i)})=Sp(u^{(i)})\cup Sp(v^{(i)})\cup \{k-l: (l,k)\in \Gamma\},$$
where $Sp(u)$ denote the set of $\{\mid i-j\mid: u_i=u_j=1\}$ and $\Gamma$ is some subset of
$$\{0,1,\dotsc, 2^{2k-2}-1\}\times \{(i+1)s,\dotsc, (i+1)s+t\}.$$
Then
\begin{align*}
k-l\in\{2s-t,2s-t+1,\dotsc,3s+t-1\}\subset\{2^{2k+1},2^{2k-1}+1,\dotsc,2^{2k+2}-1\}\subset P,
\end{align*}
which implies that $[w^{(i)}]_P\neq\emptyset$ for $i=1,2$.
Therefore, $\sigma^{2}\times\sigma^{3}$ is transitive.
To finish the
proof it is enough to show that $\sigma\times\sigma^{2}$ is not transitive.
Let $U=V=[1]_P\times[1]_P$. Note that if $m\in P$ then $2m \notin P$.
Then we have $(\sigma_{P}\times \sigma_{P}^{2})^n(U)\cap V=\emptyset$ for any $n\in \mathbb{N}$,
which implies that $\sigma\times\sigma^{2}$ is not transitive.
\end{proof}
\end{rem}

We have defined a new kind of Furstenberg family in \cite{Chzhj-Ljian-Lj} generated by a given vector of $\mathbb{N}^r$ as following.
\begin{de}
Let $\mathbf{a}=(a_1,a_2,\dotsc,a_r)$ be a vector in $\mathbb{N}^r$.
We define \emph{the family generated by the vector $\mathbf{a}$},
denoted by $\mathcal{F}[\mathbf{a}]$, as
\begin{align*}
\{F\subset\mathbb{N}:\ & \text{for every } \mathbf{n}=(n_1, n_2, \dotsc,n_{r})\in\mathbb{Z}_+^r,
\text{ there exists }k\in\mathbb{N}\text{ such that } k\mathbf{a}+\mathbf{n}\in F^r\},
\end{align*}
where $k\mathbf{a}+\mathbf{n}=(ka_1 +n_1, ka_2+n_2,\dotsc,ka_r+n_{r})$.
\end{de}
Using the family $\mathcal{F}[\mathbf{a}]$, we obtain the following characterization of multi-transitivity
with respect to $\mathbf{a}$.
\begin{thm}[\cite{Chzhj-Ljian-Lj}]\label{thm:trans-vetor-equ-condi}
Let $(X,f)$ be a dynamical system and $\mathbf{a}=(a_1,a_2,\dotsc,a_r)\in \mathbb{N}^r$.
Then $(X,f)$ is $\mathbf{a}$-transitive if and only if it is $\mathcal{F}[\mathbf{a}]$-transitive.
\end{thm}

The following observations are of the family $\mathcal{F}[\mathbf{a}]$.
\begin{lem}\label{lem:F[a]-invariant}
For every $\mathbf{a}=(a_1,a_2,\dotsc,a_r)\in \mathbb{N}^r$,
$\mathcal{F}[\mathbf{a}]$ is a proper translation $\pm$ invariant family.
\end{lem}
\begin{proof}
Fix a vector $\mathbf{a}=(a_1,a_2,\dotsc,a_r)\in \mathbb{N}^r$.
Clearly, $\mathcal{F}[\mathbf{a}]$ a proper family,
since $\emptyset\not\in \mathcal{F}[\mathbf{a}]$
and $\mathbb{N}\in \mathcal{F}[\mathbf{a}]$.
Next we show  that $\mathcal{F}[\mathbf{a}]$ is  translation $\pm$ invariant.
Let $F\in \mathcal{F}[\mathbf{a}]$ and $n\in \mathbb{N}$.
Denote $\mathbf{n_0}:=(n,n,\ldots,n)\in \mathbb{Z}_+^{r}$.
Then for every $\mathbf{n}:=(n_1, n_2 \dotsc,n_{r})\in\mathbb{Z}_+^{r}$,
there exists $k_0\in\mathbb{N}$ such that $k_0\mathbf{a}+\mathbf{n}-\mathbf{n_0}\in \mathbb{Z}_+^{r}$.
Since $F\in \mathcal{F}[\mathbf{a}]$,
there exists $k_1\in\mathbb{N}$ such that
$k_1\mathbf{a}+(k_0\mathbf{a}+\mathbf{n}-\mathbf{n_0})\in F^r$ and so with $k=k_0+k_1$,
$k\mathbf{a}+\mathbf{n}\in F^r+\mathbf{n_0}$.
Thus $F+n\in \mathcal{F}[\mathbf{a}]$. Similarly, we can show $F-n\in \mathcal{F}[\mathbf{a}]$ .
\end{proof}

\begin{lem}
For every $r\geq 2$ and $\mathbf{a}=(a_1,a_2,\dotsc,a_r)\in \mathbb{N}^r$,
$2\mathbb{N}\not\in \mathcal{F}[\mathbf{a}]$.
\end{lem}
\begin{proof}
Case 1: $a_1,a_2,\dotsc,a_r$ are odd. Pick $\mathbf{n}=(1,2,\dotsc,r)\in\mathbb{Z}_+^r$. For every $k\in\mathbb{N}$,
one of $ka_1+1$ and $ka_2+2$ must be odd. Then $2\mathbb{N}\not\in \mathcal{F}[\mathbf{a}]$.

Case 2: there is at least one even integer in $a_1,a_2,\dotsc,a_r$.
Pick $\mathbf{n}=(1,1,\dotsc,1)\in\mathbb{Z}_+^r$. For every $k\in\mathbb{N}$,
$k\mathbf{a}+\mathbf{n}$ contains some odd component. Then $2\mathbb{N}\not\in \mathcal{F}[\mathbf{a}]$.
\end{proof}

\begin{thm}\label{lem:trans-vetor-equ-condi}
Let $(X,f)$ be a dynamical system and $\mathbf{a}=(a_1,a_2,\dotsc,a_r)\in \mathbb{N}^r$.
Then the following conditions are equivalent.
\begin{enumerate}
  \item $(X,f)$ is $\mathbf{a}$-transitive;
  \item $(X,f)$ is $\mathcal{F}[\mathbf{a}]$-transitive;
  \item $(X,f)$ is transitive and  $\mathcal{F}[\mathbf{a}]$-central.
\end{enumerate}
\end{thm}
 \begin{proof}
(1) $\Leftrightarrow (2)$ follows from Theorem~\ref{thm:trans-vetor-equ-condi}.

(2) $\Leftrightarrow$ (3) follows from Lemmas~\ref{lem:F-trans-equ-F-center-trans} and~\ref{lem:F[a]-invariant}.
\end{proof}

Now applying results in section 3,
we show that multi-transitivity can be also characterized by $\mathcal {F}$-point transitivity.
\begin{thm}\label{thm:trans-vector-equ-F-point-trans}
Let $(X,f)$ be a dynamical system and $\mathbf{a}=(a_1,a_2,\dotsc,a_r)\in \mathbb{N}^r$.
Then the following conditions are equivalent.
\begin{enumerate}
  \item $(X,f)$ is $\mathbf{a}$-transitive;
  \item $(X,f)$ is $\nabla(\mathcal{F}[\mathbf{a}])$-point transitive;
  \item $Trans(X,f)=Trans_{\nabla(\mathcal{F}[\mathbf{a}])}(X,f)\neq\emptyset$.
\end{enumerate}
\end{thm}
\begin{proof}
(1) $\Rightarrow$ (2) follows from Theorem~\ref{lem:trans-vetor-equ-condi} and
Proposition~\ref{prop:tans-F-center-equi-nablaF-trans}.

(2) $\Rightarrow$ (3) follows from Proposition~\ref{prop:tans-F-center-equi-nablaF-trans}.

(3) $\Rightarrow$ (1)
By Proposition~\ref{prop:tans-F-center-equi-nablaF-trans},
$(X,f)$ is transitive and  $\mathcal{F}[\mathbf{a}]$-central.
Then by Lemmas~\ref{lem:F-trans-equ-F-center-trans} and ~\ref{lem:F[a]-invariant},
$(X,f)$ is $\mathcal{F}[\mathbf{a}]$-transitive.
Therefore by Theorem~\ref{lem:trans-vetor-equ-condi}, $(X,f)$ is $\mathbf{a}$-transitive.
\end{proof}

Let $\mathcal{F}[\infty]=\bigcap\limits_{i=1}^\infty \mathcal{F}[\mathbf{a}_i]$,
where $\mathbf{a}_i=(1,2,\dotsc,i)$ for $i\in\mathbb{N}$.
Using the family $\mathcal{F}[\infty]$, we have the following characterization of multi-transitivity and
strong multi-transitivity.
\begin{thm}[\cite{Chzhj-Ljian-Lj}]\label{thm:multi-trans-equ-condi-infty}
Let $(X,f)$ be a dynamical system. Then
\begin{enumerate}
  \item $(X,f)$ is multi-transitive if and only if it is $\mathcal{F}[\infty]$-transitive;
  \item $(X,f)$ is strongly multi-transitive if and only if it is $\mathcal{F}[\infty]$-mixing.
\end{enumerate}
\end{thm}

By Lemma~\ref{lem:F[a]-invariant},
we have that $\mathcal{F}[\infty]$ is also transition $\pm$ invariant.
Then we obtain a new characterization of multi-transitivity.

\begin{thm}\label{thm:multrans-nablaF(infty)-poin-trans-equi}
Let $(X,f)$ be a dynamical system. Then the following conditions are equivalent.
\begin{enumerate}
\item $(X,f)$ is multi-transitive;
\item $(X,f)$ is $\mathcal{F}[\infty]$-transitive;
\item $(X,f)$ is $\nabla(\mathcal{F}[\infty])$-point transitive.
\end{enumerate}
\end{thm}

\begin{proof}
(1) $\Leftrightarrow$ (2) follows from the definitions 
and Theorem~\ref{thm:multi-trans-equ-condi-infty}.

(2) $\Leftrightarrow $ (3) folows from Proposition~\ref{lem:F-trans-equi-nablaF-trans}
and the fact that $\mathcal{F}[\infty]$ is transition $+$ invariant.
\end{proof}

Let $\mathcal{F}_{smt}=\nabla(\mathcal{F}[\infty]\cap\mathcal{F}_t)$.
Then we obtain the following characterization of strongly multi-transitivity.

\begin{thm}\label{cor:str-multi-equi}
Let $(X,f)$ be a dynamical system. Then the following conditions are equivalent.
\begin{enumerate}
\item $(X,f)$ is strongly multi-transitive;
\item $(X,f)$ is weakly mixing and multi-transitive;
\item $(X,f)$ is $\mathcal{F}[\infty]$-mixing;
\item $(X,f)$ is $\mathcal{F}_{smt}$-point transitive.
\end{enumerate}
\end{thm}

\begin{proof}
(1) $\Rightarrow$ (2) is trivial.

(2) $\Rightarrow$ (3) Let $U_1$, $U_2$, $V_1$ and $V_2$ be non-empty open subsets of $X$.
Since $(X,f)$ is weakly mixing, there exists non-empty open subsets $U$, $V$ of $X$ such that
$N_{f}(U,V)\subset N_{f}(U_1,V_1)\cap N_{f}(U_2,V_2)$.
Since $(X,f)$ is multi-transitive, by Theorem~\ref{thm:multrans-nablaF(infty)-poin-trans-equi},
$N_{f}(U,V)\in \mathcal{F}[\infty]$.
Hence, $N_{f}(U_1,V_1)\cap N_{f}(U_2,V_2)\in \mathcal{F}[\infty]$.

(3) $\Rightarrow$ (1) follows from Theorem~\ref{thm:multi-trans-equ-condi-infty}.

(2) $\Leftrightarrow$ (4) follows from Theorems~\ref{thm:weak-str-mix-nablaF-poin-trans},
\ref{thm:multrans-nablaF(infty)-poin-trans-equi}, and the fact that $\nabla(\mathcal{F}[\infty])\cap\nabla(\mathcal{F}_t)=\nabla(\mathcal{F}[\infty]\cap\mathcal{F}_t)$.
\end{proof}

\section[Delta-transitivity with respect to a vector]{\texorpdfstring{$\Delta$}{Delta}-transitivity with respect to a vector}

In this section, we first discuss some propositions of $\Delta$-transitivity
and then provide a characterization of $\Delta$-transitivity via $\mathcal{F}$-point transitvity.

\begin{de}
Let $(X, f)$ be a dynamical system and $\mathbf{a}=(a_{1}, a_{2}, \dotsc, a_r)$ be a vector in $\mathbb{N}^r$.
We say that the system $(X,f)$ is
\begin{enumerate}
\item \emph{$\Delta$-transitive with respect to the vector $\mathbf{a}$}
(or briefly \emph{$\Delta$-$\mathbf{a}$-transitive})
if there exists a point $x\in X$ such that
$(x,x,\dotsc,x)$ is a transitive point of $(X^r,  f^{(\mathbf{a})})$.
\item \emph{$\Delta$-transitive}
if it is $\Delta$-transitive with respect to $(1,2,\dotsc,n)$ for all $n\in\mathbb N$.
\end{enumerate}
\end{de}

\begin{rem}Let $(X,f)$ be a dynamical system.
If $(X,f)$ is $\Delta$-transitive with respect to $(1,1)$,
then there exists a point $x\in X$ such that $\{(f^n(x), f^n(x)):\ n\in \mathbb{N}\}$ is dense in $X\times X$.
But $\{(f^n(x), f^n(x)):\ n\in \mathbb{N}\}$ is a subset of the diagonal of $X\times X$.
This implies that $X$ must be a singleton.
For this reason, we only discuss the $\Delta$-$\mathbf{a}$-transitive system,
where $\mathbf{a}\in \mathbb{N}^r_*=\{(n_1,n_2,\dotsc,n_r)\in\mathbb{N}^r:\ n_1<n_2<\dotsb<n_r\}$.
\end{rem}

\begin{prop}\label{prop:triangle-trans-equ-condi}
Let $(X,f)$ be a dynamical system and $\mathbf{a}=(a_1,a_2,\dotsc,a_r)\in\mathbb{N}_*^r$.
Then the following conditions are equivalent.
\begin{enumerate}
\item $(X,f)$ is $\Delta$-$\mathbf{a}$-transitive;
\item For every non-empty open subsets $U_0$, $U_1, \dotsc, U_r$ of $X$,
there exists some $n\in \mathbb{N}$ such that
\[U_0\cap\bigcap_{i=1}^{r} f^{-na_i}(U_i)\neq\emptyset;\]
\item $Trans(X^r,f^{(\mathbf{a})})\cap \Delta_{X^r}\neq\emptyset$ is
 a dense $G_\delta$ subset of $\Delta_{X^r}$.
\end{enumerate}
\end{prop}
\begin{proof}
(1) $\Rightarrow$ (2) Let $U_0, U_1, \dotsc, U_r$ be non-empty open subsets of $X$
and $(x,x,\dotsc,x)$ be a transitive point of $(X^r,f^{(\mathbf{a})})$.
Choose $k\in \mathbb{N}$ such that $y=f^{k}(x)\in U_0$.
As $f^{a_1}\times f^{a_2}\times\dotsb \times f^{a_r}$ commutes with $f\times f\times\dotsb \times f$,
we have $(y,y,\dotsc,y)$ is also a transitive point of $(X^r,f^{(\mathbf{a})})$.
Then there exists some $n\in \mathbb{N}$ such that
$(f^{na_1}(y), f^{na_2}(y),\dotsc,$ $f^{na_r}(y))\in U_1\times U_2\times\dotsb \times U_r$,
that is
\[y\in U_0\cap\bigcap_{i=1}^{r} f^{-na_i}(U_i).\]

(2) $\Rightarrow$ (3) Let $\{B_k: k\in \mathbb{N}\}$ be a countable base of open balls of $X$.
Put
$$Y=\bigcap_{(k_1,k_2,\dotsc,k_r)\in \mathbb{N}^r}\bigcup_{n\in \mathbb{N}} \bigcap _{i=1}^{r}f^{-na_i}(B_{k_i}) $$
The set $\bigcup_{n\in \mathbb{N}} \bigcap _{i=1}^{r}f^{-na_i}(B_{k_i})$ is clear open, and it is dense by (2).
Then by Baire category theorem, $Y$ is a dense $G_\delta$ subset of $X$.
Let $x\in Y$ and $W$ be a non-empty open subset of $X^r$.
There exists $(k_1,k_2,\dotsc,k_r)\in \mathbb{N}^r$ such that
$B_{k_1}\times B_{k_2}\times\dotsb \times B_{k_r}\subset W$.
By the construction of $Y$, there exists $n\in \mathbb{N}$ such that $x\in \bigcap _{i=1}^{r}f^{-na_i}(B_{k_i})$,
that is $(f^{na_1}(x), f^{na_2}(x),\dotsc,$ $f^{na_r}(x))\in B_{k_1}\times B_{k_2}\times\dotsb \times B_{k_r}\subset W$.
Then $\{(f^{na_1}(x), f^{na_2}(x),\dotsc,$ $f^{na_r}(x)):n\in \mathbb{N}\}$ is dense in $X^r$,
i.e., $(x,x,\dotsc,x)$ is a transitive point of $(X^r,f^{(\mathbf{a})})$.
Therefore, $Trans(X^r,f^{(\mathbf{a})})\cap \Delta_{X^r}\neq\emptyset$ is a dense $G_\delta$ subset of $\Delta_{X^r}$.

(3) $\Rightarrow$ (1) is obvious.
\end{proof}

\begin{prop}\label{pro:delta-tran-imply-product-deta-trans}
Let $(X,f)$ be a dynamical system. If $(X,f)$ is $\Delta$-transitive, then for every
$r\in\mathbb{N}$ and  $\mathbf{a}\in\mathbb{N}^r_*$, $(X^r,f^{(\mathbf{a})})$ is also $\Delta$-transitive.
\end{prop}
\begin{proof}
Let $r\in\mathbb{N}$ and $\mathbf{a}=(a_1,a_2,\dotsc,a_r)\in\mathbb{N}_*^r$.
Fix an integer $n\in\mathbb{N}$.
Now we are going to prove that
$(X^{r},f^{(\mathbf{a})})$ is $\Delta$-transitive with respect to the vector $(1,2,\ldots,n)$.
Let $U_1^{(0)},U_2^{(0)},\ldots,$ $U_r^{(0)}$,
$U_1^{(1)},U_2^{(1)},\ldots,U_r^{(1)}$, $\ldots$ , $U_1^{(n)},U_2^{(n)},\ldots,U_r^{(n)}$
be non-empty open subsets of $X$. For every $k\in\{0,1,2,\ldots, n(\Sigma_{t=1}^{r}a_t)+na_r\}$,
we define a non-empty open subset $V_k$ of $X$ as
$$V_{k}:=\left\{\begin{array}{ll}
U_i^{(j)}, & \text{when }k=n(\Sigma_{t=1}^{i}a_t)+ja_i \text{ for some }0\leq j\leq n \text{ and } 1\leq i\leq r, \\
X, & \text{otherwise}.
\end{array}\right.$$
We first show that every $V_k$ is well defined.
If $k$ can not be represented as the form of
$n(\Sigma_{t=1}^{i}a_t)+ja_i$, then $V_{k}=X$.
If $k$ can be represented as the form of
$n(\Sigma_{t=1}^{i}a_t)+ja_i$, then we want to show that this representation is unique.
Suppose that there exist $i_1, i_2\in\{1,2,\ldots r\}$ and $j_1,j_2\in \{0,1,\ldots, n\}$ such that
$n(\Sigma_{t=1}^{i_1}a_t)+j_1a_{i_1}=n(\Sigma_{t=1}^{i_2}a_t)+j_2a_{i_2}$.
If $i_1>i_2$, then $n(\Sigma_{t=i_2+1}^{i_1}a_t)+j_1a_{i_1}=j_2a_{i_2}$.
By the fact $n\geq j_2$ and $a_{i_2+1}> a_{i_2}$, we have
$n(\Sigma_{t=i_2+1}^{i_1}a_t)+j_1a_{i_1}\geq na_{i_2+1}+j_1a_{i_1}>j_2a_{i_2}$,
which is a contradiction.
Then $i_1=i_2$.
Hence $j_1a_{i_1}=j_2a_{i_2}$ and then  $j_1=j_2$.

By proposition~\ref{prop:triangle-trans-equ-condi},
there exists some $m\in\mathbb{N}$ such that
$$V_0\cap\bigcap_{i=1}^{n(\Sigma_{t=1}^{r}a_t)+na_r}f^{-im}(V_i)\neq\emptyset,$$
which implies that
$$f^{-mn(\Sigma_{t=1}^{i}a_t)}V_{n(\Sigma_{t=1}^{i}a_t)}
\cap\bigcap_{j=0}^{n}f^{-m(n(\Sigma_{t=1}^{i}a_t)+ja_i)}(V_{n(\Sigma_{t=1}^{i}a_t)+ja_i})\neq\emptyset, \text{ for }i=1,2,\ldots,r. $$
Then
$$U_i^{(0)}\cap\bigcap_{j=0}^{n}f^{-mja_i}(U_i^{(j)})\neq\emptyset, \text{ for }i=1,2,\ldots,r,$$
and
$$(U_1^{(0)}\times U_2^{(0)}\times\cdots\times U_r^{(0)})
\cap\bigcap_{j=0}^{n}(f^{(\mathbf{a})})^{-mj}(U_1^{(j)}\times U_2^{(j)}\times\cdots\times U_r^{(j)})\neq\emptyset.$$
By Proposition~\ref{prop:triangle-trans-equ-condi},
$(X^r,f^{(\mathbf{a})})$ is $\Delta$-transitive with respect to the vector $(1,2,\ldots,n)$.
\end{proof}

The following result is essentially contained in~\cite{Moothathu-2010}.
For completeness, we provide a proof.
\begin{lem}
Let $(X,f)$ be a dynamical system.
If $(X,f)$ is $\Delta$-transitive with respect to $(1,2)$,
then it is weakly mixing.
\end{lem}

\begin{proof}
Let $U$ and $V$ be two non-empty open subsets of $X$.
To show that $(X,f)$ is weakly mixing, it suffices to show that $N(U,V)\cap N(V,V)\neq\emptyset$.
By Proposition~\ref{prop:triangle-trans-equ-condi},
there exists $n\in\mathbb{N}$ such that
\[U\cap f^{-n}(V)\cap f^{-2n}(V)\neq\emptyset.\]
Then $U\cap f^{-n}(V)\neq\emptyset$ and $V\cap f^{-n}(V)\neq\emptyset$,
that is $n\in N(U,V)\cap N(V,V)$.
\end{proof}

\begin{ex}
For every $p>2$, there exists a subshift which is $\Delta$-transitive with respect to $(1,2)$
but not with respect to $(1,p)$.
\end{ex}
\begin{proof}
Let $\Sigma_{2}=\bigl\{x=(x_{n})_{n=-\infty}^{\infty}:x_{n}\in\{0, 1\}\bigl\}$.
The two-sided shift $\sigma:\Sigma_{2}\rightarrow\Sigma_{2}$ is defined by
\[\sigma(\dotsc x_{-2}x_{-1}.x_{0}x_{1}x_{2}\dotsc)=\dotsc x_{-2}x_{-1}x_{0}.x_{1}x_{2}\dotsc,\]
which is a homeomorphism.
We say that a word $w=w_1w_2\dotsc w_n$ appears in $x\in \Sigma_{2}$
if there exists $j\in\mathbb{Z}$ such that $x_{j+i}=w_i$ for $i=1,2,\dotsc,n$.
Let $p>2$. Put $F=\{11\}\cup \{1u1v1:(p-1)(|u|+1)=|v|+1\}$ and
$X=\{x\in\Sigma_{2}:\text{ if }u\text{ appears in } x,\text{ then } u\text{ is not in } F\}$.
It is clear that $X$ is closed and $\sigma$-invariant.
We are going to show that $(X, \sigma|_{X})$ is $\Delta$-transitive with respect to $(1,2)$ but not respect to $(1,p)$.

Let $x$, $y$, $z\in X$ and $k\in \mathbb{N}$.
Let $u=x_{-k}\dotsc x_{0}\dotsc x_{k}$,
$v=y_{-k}\dotsc y_{0}\dotsc y_{k}$ and $w=z_{-k}\dotsc $ $z_{0}\dotsc z_{k}$.
Put $[u]=\{x'\in X: x'_{-k}\dotsc x'_{0}\dotsc x'_{k}=u\},  [v]=\{x'\in X: x'_{-k}\dotsc x'_{0}\dotsc x'_{k}=v\}$
and $[w]=\{x'\in X: x'_{-k}\dotsc x'_{0}\dotsc x'_{k}=w\}$.
In the product topology, $[u],[v]$ and $[w]$ are basic neighborhoods of $x, y$ and $z$ respectively.
Consider a point $r(n)\in \Sigma_{2}$ defined as $r(n)=0^{\infty}u0^{n-2k-1}v0^{n-2k-1}w0^{\infty}$.
Then if $n$ is large enough, ones have  $r(n)\in X$ and $r(n)\in[u]\cap\sigma^{-n}[v]\cap\sigma^{-2n}[w]$.
By Proposition~\ref{prop:triangle-trans-equ-condi},
$(X, \sigma|_{X})$ is $\Delta$-transitive with respect to $(1,2)$.

Now we want to show that $(X,\sigma|_{X})$ is not $\Delta$-transitive with respect to $(1,p)$.
Suppose that there exists some $x\in X$ such that $\{(\sigma\times
\sigma^{p})^{n}(x, x):n\in \mathbb{N}\}$ is dense in $X^{2}$.
Let $W=\{y\in
X:y_{0}=1\}$.
Pick a $k\in\mathbb{N}$ such that $y=\sigma^{k}(x)\in W$.
Since  $\sigma\times\sigma^{p}$ commutes with $\sigma\times\sigma$,
the orbit of $(y, y)$ under $\sigma\times\sigma^{p}$ is dense in $X^2$.
Then there exists some $n\in\mathbb{N}$ such that $(\sigma^{n}(y),\sigma^{pn}(y))\in
W\times W$.
Thus, $y, \sigma^{n}(y), \sigma^{pn}(y)\in W$, that is $y_{0}=y_{n}=y_{pn}=1$.
This contradicts to the construction of $X$.
Then by Proposition~\ref{prop:triangle-trans-equ-condi},
$(X, \sigma|_{X})$ is not $\Delta$-transitive with respect to $(1,p)$.
\end{proof}

\begin{rem}
The full shift is $\Delta$-transitive and
there exists a strongly mixing system which is not $\Delta$-transitive~\cite{Moothathu-2010}.
Then by Proposition~\ref{pro:Not-F-trans-iter},
the collection of $\Delta$-transitive systems can not be classified by $\mathcal{F}$-transitivity.
\end{rem}

\begin{thm}\label{thm:triangle-trans-vector-Family}
Let $(X,f)$ be a dynamical system and $\mathbf{a}=(a_1,a_2,\dotsc,a_r)\in\mathbb{N}_*^r$.
Then the following conditions are equivalent.
\begin{enumerate}
\item $(X,f)$ is $\Delta$-$\mathbf{a}$-transitive;
\item  $(X,f)$ is $\mathcal{F}[\mathbf{a}]$-point transitive;
\item $Trans_{\mathcal{F}[\mathbf{a}]}(X,f)$ is residual in $X$.
\end{enumerate}
\end{thm}
\begin{proof}
(1) $\Rightarrow$ (3) Assume that $(X,f)$ is $\Delta$-$\mathbf{a}$-transitive.
Let $\{B_k: k\in \mathbb{N}\}$ be a countable base of open balls of $X$.
Put
$$Y=\bigcap_{(k_1,k_2,\dotsc,k_r)\in \mathbb{N}^r}\bigcup_{n\in \mathbb{N}} \bigcap _{i=1}^{r}f^{-na_i}(B_{k_i}).$$
By the proof of Proposition~\ref{prop:triangle-trans-equ-condi},
$Y$ is a dense $G_\delta$ subset of $X$.
Now it suffices to show that every point in $Y$ is an $\mathcal{F}[\mathbf{a}]$-transitive point.
Let $x\in Y$, $U$ be a non-empty open subset of $X$ and $n_1,n_2,\dotsc,n_r\in\mathbb{Z}_+$.
There exists $(k_1,k_2,\dotsc,k_r)\in \mathbb{N}^r$ such that
$$B_{k_1}\times B_{k_2}\times\dotsb\times B_{k_r}\subset f^{-n_1}(U)\times f^{-n_2}(U)\times\dotsb\times f^{-n_r}(U).$$
By the construction of $Y$,
there is some $k\in \mathbb{N}$ such that $x\in \bigcap _{i=1}^{r}f^{-ka_i}(B_{k_i})$,
then
$$\bigl(f^{a_1k}(x), f^{a_2k}(x),\dotsc,f^{a_rk}(x)\bigl)\in B_{k_1}\times B_{k_2}\times\dotsb\times B_{k_r}
\subset f^{-n_1}(U)\times f^{-n_2}(U)\times\dotsb\times f^{-n_r}(U),$$
which implies that $\{a_1k+n_1,a_2k+n_2,\dotsc,a_rk+n_r\}\subset N(x,U)$.

(3) $\Rightarrow$ (2) is obvious.

(2) $\Rightarrow$ (1)
Let $U_1$, $U_2,\dotsc, U_r$ be non-empty open subsets of $X$ and $x\in Trans_{\mathcal{F}[\mathbf{a}]}(X,f)$.
Applying the transitivity of $(X,f)$ to the two non-empty open sets $U_{r-1}$ and $U_r$,
we pick an $\ell_{r-1}\in\mathbb{N}$ such that
$$U_{r-1}\cap f^{-\ell_{r-1}}(U_r)\neq\emptyset.$$
Now applying the transitivity of $(X,f)$ to the two non-empty open sets
$U_{r-2}$ and $U_{r-1}\cap f^{-\ell_{r-1}}(U_r)$,
we pick an $\ell_{r-2}\in\mathbb{N}$ such that
$$U_{r-2}\cap f^{-\ell_{r-2}}(U_{r-1}\cap f^{-\ell_{r-1}}(U_r))\neq\emptyset.$$
After repeating this process $(r-1)$ times,
we obtain a sequence $\{\ell_i\}_{i=1}^{r-1}$ of positive integers and
\begin{align*}
U:=&U_1\cap f^{-\ell_{1}}\Bigl(U_{2}\cap f^{-\ell_{2}}\bigl(\dotsb\cap(U_{r-1}\cap f^{-\ell_{r-1}}(U_r))\bigr)\Bigr)\\
=& U_1\cap f^{-\ell_{1}}(U_{2})\cap f^{-(\ell_{1}+\ell_{2})}(U_{3})\cap\dotsb\cap
 f^{-(\ell_{1}+\ell_{2}+\dotsb+\ell_{r-1})}(U_r)\neq\emptyset.
\end{align*}
Put $\ell:=\ell_{1}+\ell_{2}+\dotsb+\ell_{r-1}$.
Since $N(x,U)\in\mathcal{F}[\mathbf{a}]$,
for positive integers $a_1\ell$,  $a_2\ell-\ell_{1}$,
$a_3\ell-\ell_{1}-\ell_{2}, \dotsc, a_r\ell-\ell_{1}-\ell_{2}-\dotsb-\ell_{r-1}$,
there exists  $k\in\mathbb{N}$ such that
$$\{a_1k+a_1\ell, a_2k+a_2\ell-\ell_{1},\dotsc, a_rk+a_r\ell-\ell_{1}-\ell_{2}-\dotsb-\ell_{r-1}\}\subset N(x,U).$$
Then
\begin{align*}
  &f^{{a_1k+a_1\ell}}(x)\in U_1,\\
  &f^{a_2k+a_2\ell-\ell_{1}}(x)\in f^{-\ell_{1}}(U_2),\\
  &\dotsc \\
  &f^{a_rk+a_r\ell-\ell_{1}-\ell_{2}-\dotsb-\ell_{r-1}}(x)\in f^{-(\ell_{1}+\ell_{2}+\dotsb+\ell_{r-1})}(U_r),
\end{align*}
which imply that
$$(f^{a_1(k+\ell)}(x),f^{a_2(k+\ell)}(x),\dotsc,f^{a_r(k+\ell)}(x))\in U_1\times U_2\times\dotsb\times U_r.$$
Thus $(x,x,\dotsc,x)$ is a transitive point of $(X^r,f^{(\mathbf{a})})$.
\end{proof}

Recall that $\mathcal{F}[\infty]=\bigcap\limits_{i=1}^\infty \mathcal{F}[\mathbf{a}_i]$,
where $\mathbf{a}_i=(1,2,\dotsc,i)$ for $i\in\mathbb{N}$.
We have the following characterization of $\Delta$-transitivity.

\begin{thm}
Let $(X,f)$ be a dynamical system.
Then the following conditions are equivalent.
\begin{enumerate}
\item $(X,f)$ is $\Delta$-transitive;
\item  $(X,f)$ is $\mathcal{F}[\infty]$-point transitive;
\item $Trans_{\mathcal{F}[\infty]}(X,f)$ is residual in $X$.
\end{enumerate}

\end{thm}
\begin{proof}
(1) $\Rightarrow$ (3) For every $\mathbf{a}_i$, $(X,f)$ is $\Delta$-transitive with respect to $\mathbf{a}_i$.
By Theorem~\ref{thm:triangle-trans-vector-Family},
$Trans_{\mathcal{F}[\mathbf{a}_i]}(X,f)$ is residual in $X$ for every $\mathbf{a}_i$.
Then $\bigcap_{i=1}^{\infty}Trans_{\mathcal{F}[\mathbf{a}_i]}(X,f)$ is also residual in $X$.
Now the result follows from the fact
$\bigcap_{i=1}^{\infty}Trans_{\mathcal{F}[\mathbf{a}_i]}(X,f)=Trans_{\mathcal{F}[\infty]}(X,f)$.

(3) $\Rightarrow$ (2) is obvious.

(2) $\Rightarrow$ (1)
Let $U$ be a non-empty open subset of $X$ and $x\in Trans_{\mathcal{F}[\infty]}(X,f)$.
Then $N(x,U)\in\mathcal{F}[\infty]$, which implies that $N(x,U)\in\mathcal{F}[\mathbf{a}_i]$ for every $i\in\mathbb{N}$.
Hence $(X,f)$ is $\mathcal{F}[\mathbf{a}_i]$-point transitive for every $i\in\mathbb{N}$.
By Theorem~\ref{thm:triangle-trans-vector-Family}, it is $\Delta$-transitive.
\end{proof}

Let $(X,f)$ be a dynamical system.
Following ~\cite{Oprocha-Zhang-2011}, we say that a subset $A$ of $X$ is a \emph{transitive set} if
for every pair of non-empty open subsets $U, V$ of $X$ intersecting $A$, the set $N(A \cap U, V )$ is non-empty;
a recurrent set if for every non-empty open subset $U$ of $X$ intersecting $A$, the set $N(A \cap U,U)$ is non-empty.
It is not hard to see that a closed subset $A$ of $X$ is a transitive set if and only if
$(\overline{Orb(A,f)},f)$ is transitive and $Trans(\overline{Orb(A,f)},f)\cap A$ is residual in $A$,
where $Orb(A,f)=\bigcup_{n=0}^\infty f^n(A)$.
The following proposition indicates that if $(X,f)$ is $\Delta$-transitive with respect to $(1,2,\dotsc,n)$
then $\Delta_{n+1}$ is a transitive set in $(X^{n+1},f^{(n+1)})$.

\begin{prop}\label{prop:delta-transitive-set}
Let $(X,f)$ be a dynamical system and $\mathbf{a}=(a_1,a_2,\dotsc,a_{r+1})\in\mathbb{N}_*^{r+1}$.
 Then $\overline{Orb(\Delta_{X^{r+1}},f^{(\mathbf{a})})} = X^{r+1}$
 if and only of $(X,f)$ is $\Delta$-$\mathbf{a}'$-transitive where $\mathbf{a}'=(a_2-a_1,\dotsc,a_{r+1}-a_1)$.
\end{prop}
\begin{proof}
Fix non-empty open subsets $U_0$, $U_1, \dotsc, U_r$ of $X$ and put $W=U_0\times U_1\times \dotsc\times U_r$.

Necessity.
Since  $\overline{Orb(\Delta_{X^{r+1}},f^{(\mathbf{a})})} = X^{r+1}$,
there exists $n\in \mathbb{N}$ such that $(f^{(\mathbf{a})})^{n}(\Delta_{X^{r+1}})\cap W\neq \emptyset$.
Thus there exists $x\in X$ such that $f^{na_i}(x)\in U_{i-1}$, $i=1,2,\ldots,r+1$.
Put $y=f^{na_1}(x)$. Then
\[y\in U_0\cap\bigcap_{i=1}^{r} f^{-n(a_{i+1}-a_1)}(U_i).\]
By Proposition~\ref{prop:triangle-trans-equ-condi},
 $(X,f)$ is $\Delta$-$\mathbf{a}'$-transitive where $\mathbf{a}'=(a_2-a_1,\dotsc,a_{r+1}-a_1)$.

Sufficiency.
Since $(X,f)$ is $\Delta$-$\mathbf{a}'$-transitive where $\mathbf{a}'=(a_2-a_1,\dotsc,a_{r+1}-a_1)$,
by Proposition~\ref{prop:triangle-trans-equ-condi},
there exists $n\in \mathbb{N}$ and
\[y\in U_0\cap\bigcap_{i=1}^{r} f^{-n(a_{i+1}-a_1)}(U_i).\]
Since the map $f^{na_1}$ is surjective,
there exists $x\in X$ with $f^{na_1}(x)=y$ and then
\[x\in \bigcap_{i=0}^{r} f^{-na_{i+1}}(U_i),\]
that is $(f^{(\mathbf{a})})^{n}(\Delta_{X^{r+1}})\cap W\neq \emptyset$.
Then $\overline{Orb(\Delta_{X^{r+1}},f^{(\mathbf{a})})} = X^{r+1}$,
since $U_0$, $U_1, \dotsc, U_r$ are arbitrary.
\end{proof}

\section{Multi-minimality with respect to a vector}
\begin{de}
Let $(X,f)$ be a dynamical system and $\mathbf{a}=(a_1,a_2,\dotsc,a_r)\in\mathbb{N}^r_*$.
We say that the  system $(X,f)$ is
\begin{enumerate}
\item \emph{multi-minimal with respect to the vector $\mathbf{a}$}
(or briefly \emph{$\mathbf{a}$-minimal})
if the product system $(X^r,f^{(\mathbf{a})})$ is minimal.
\item \emph{multi-minimal} if
it is  multi-minimal with respect to $(1,2,\dotsc,n)$ for every $n\in\mathbb{N}$.
\end{enumerate}
\end{de}

The fourth question proposed in \cite{D-Kwietniak-P-Oprocha-2010} is the following:
\begin{ques}
Is there any non-trivial characterization of multi-minimality in terms of some
dynamical properties?
\end{ques}

In this section, we will answer this question by providing a characterization of multi-minimality
by $\mathcal{F}$-point transitivity.
First, we show some basic properties of multi-minimality.
\begin{prop}\label{pro:multimini-equi-k-iterate-multi-mini}
Let $(X,f)$ be a dynamical system.
Then $(X,f)$ is multi-minimal if and only if for every $r\in\mathbb{N}$,
$(X,f^r)$ is multi-minimal.
\end{prop}
\begin{proof}
The sufficiency is obvious. We are left to show the necessity.
Fix an integer $r\in\mathbb{N}$.
We are going to prove that $(X,f^r)$ is multi-minimal.
Let $n\in\mathbb{N}$ and $\mathbf{a}_n=(1,2,\dotsc,n)$.
It suffices to show that $(X,f^r)$ is multi-minimal respect to $\mathbf{a}_n$, i.e.,
to show that $(X^n,(f^{r})^{(\mathbf{a}_n)})$ is minimal.
Since $(X,f)$ is multi-minimal,
$(X,f)$ is multi-minimal respect to $\mathbf{a}_{rn}$, i.e.,
$(X^{nr},f\times f^2\times\cdots\times f^{nr})$ is minimal.
Thus $(X^n,(f^{r})^{(\mathbf{a}_n)})$ is minimal, since it is a factor of $(X^{nr},f^{(\mathbf{a}_{nr})})$.
\end{proof}

\begin{rem}
The POD  (proximal orbit dense) flows are examples of non-periodic 
multi-minimal systems (see~\cite{D-Kwietniak-P-Oprocha-2010,POD-flow}).
There exists a strongly mixing system which is not $\Delta$-transitive~\cite{Moothathu-2010},
then this system is also not multi-minimal.
Hence by Proposition~\ref{pro:Not-F-trans-iter},
the collection of multi-minimal systems can not be classified by $\mathcal{F}$-transitivity.
\end{rem}

\begin{de}
Let $\mathbf{a}=(a_1,a_2,\dotsc,a_r)$ be a vector in $\mathbb{N}^r$.
We define \emph{the family generated by the vector $\mathbf{a}$ and $\mathcal{F}_{s}$},
denoted by $\mathcal{F}_s[\mathbf{a}]$, as
\begin{align*}
\{F\subset\mathbb{N}:\ &\text{for every }n_1, n_2, \dotsc,n_{r}\in\mathbb{Z}_+,\text{ there exists a syndetic set }F'\subset\mathbb{N}\\
&\text{such that }a_1F'+n_1,a_2F'+n_2,\dotsc,a_rF'+n_{r}\subset F \}.
\end{align*}
\end{de}
Using the family $\mathcal{F}_s[\mathbf{a}]$, we have the following characterization of multi-minimality
with respect to $\mathbf{a}$.

\begin{thm}\label{thm:mul-minimal-vector-Family}
Let $(X,f)$ be a dynamical system and $\mathbf{a}=(a_1,a_2,\dotsc,a_r)\in\mathbb{N}_*^r$.
Then the following conditions are equivalent.
\begin{enumerate}
\item $(X,f)$ is multi-minimal with respect to $\mathbf{a}$;
\item  $(X,f)$ is $\mathcal{F}_s[\mathbf{a}]$-point transitive;
\item $Trans_{\mathcal{F}_s[\mathbf{a}]}(X,f)=X$.
\end{enumerate}
\end{thm}
\begin{proof}
(1) $\Rightarrow$ (3) Assume that $(X,f)$ is multi-minimal with respect to $\mathbf{a}$.
Fix a point $x\in X$. We want to show that $x$ is an $\mathcal{F}_s[\mathbf{a}]$-transitive point.
Let $U$ be a non-empty open subset of $X$ and $n_1, n_2, \dotsc, n_r \in \mathbb{Z}_+$.
Since $(X^r,f^{(\mathbf{a})})$ is minimal, by Lemma~\ref{lem:minimal-Fs} $(x,x,\dots,x)$
is an $\mathcal{F}_s$-transitive point of $(X^r,f^{(\mathbf{a})})$.
Then there exists a syndetic set $F$ such that
\begin{align*}
  F&\subset N_{f^{(\mathbf{a})}}((x,x,\dotsc,x),f^{-n_1}(U)\times f^{-n_2}(U)\times\dotsb\times f^{-n_r}(U))\\
  &=N_{f^{a_1}}(x,f^{-n_1}(U))\cap N_{f^{a_2}}(x,f^{-n_2}(U))\times\dotsb\times\cap  N_{f^{a_r}}(x,f^{-n_r}(U)),
\end{align*}
that is
$$a_1F+n_1, a_2F+n_2,\dotsc, a_rF+n_r\subset N(x,U).$$
Thus $N(x,U)\in \mathcal{F}_s[\mathbf{a}]$ and  $x$ is an $\mathcal{F}_s[\mathbf{a}]$-transitive point of $(X,f)$.

(3) $\Rightarrow$ (2) is obvious.

(2) $\Rightarrow$ (1) Assume that $(X,f)$ is $\mathcal{F}_s[\mathbf{a}]$-point transitive.
Let $U_1$, $U_2,\dotsc, U_r$ be non-empty open subsets of $X$ and $x\in Trans_{\mathcal{F}_s[\mathbf{a}]}(X,f)$.
Applying the transitivity of $(X,f)$ to the two non-empty open sets $U_{r-1}$ and $U_r$,
we pick an $\ell_{r-1}\in\mathbb{N}$ such that
$$U_{r-1}\cap f^{-\ell_{r-1}}(U_r)\neq\emptyset.$$
Now applying the transitivity of $(X,f)$ to the two non-empty open sets
$U_{r-2}$ and $U_{r-1}\cap f^{-\ell_{r-1}}(U_r)$,
we pick an $\ell_{r-2}\in\mathbb{N}$ such that
$$U_{r-2}\cap f^{-\ell_{r-2}}(U_{r-1}\cap f^{-\ell_{r-1}}(U_r))\neq\emptyset.$$
After repeating this process $(r-1)$ times,
we obtain a sequence $\{\ell_i\}_{i=1}^{r-1}$ of positive integers and
\begin{align*}
U:=&U_1\cap f^{-\ell_{1}}\Bigl(U_{2}\cap f^{-\ell_{2}}\bigl(\dotsb\cap(U_{r-1}\cap f^{-\ell_{r-1}}(U_r))\bigr)\Bigr)\\
=& U_1\cap f^{-\ell_{1}}(U_{2})\cap f^{-(\ell_{1}+\ell_{2})}(U_{3})\cap\dotsb\cap
 f^{-(\ell_{1}+\ell_{2}+\dotsb+\ell_{r-1})}(U_r)\neq\emptyset.
\end{align*}
Put $\ell:=\ell_{1}+\ell_{2}+\dotsb+\ell_{r-1}$.
Since $N_f(x,U)\in\mathcal{F}_s[\mathbf{a}]$,
for positive integers $a_1\ell$,
$a_2\ell-\ell_{1}$, $a_3\ell-\ell_{1}-\ell_{2}, \dotsc, a_r\ell-\ell_{1}-\ell_{2}-\dotsb-\ell_{r-1}$,
there exists a syndetic set $F$ such that
$$\{a_1F+a_1\ell, a_2F+a_2\ell-\ell_{1},\dotsc, a_rF+a_r\ell-\ell_{1}-\ell_{2}-\dotsb-\ell_{r-1}\}\subset N_f(x,U).$$
Then
\begin{align*}
  &a_1F+a_1\ell\subset N_f(x,U_1),\\
  &a_2F+a_2\ell-\ell_{1}\subset N_f(x,f^{-\ell_{1}}(U_2)),\\
  &\dotsc \\
  &a_rF+a_r\ell-\ell_{1}-\ell_{2}-\dotsb-\ell_{r-1}\subset N_f(x,f^{-(\ell_{1}+\ell_{2}+\dotsb+\ell_{r-1})}(U_r))\\
\end{align*}
which imply that
$$F+\ell\subset N_{f^{(\mathbf{a})}}((x,x,\dotsc,x), U_1\times U_2\times\dotsb\times U_r).$$
Clearly, $F+\ell$ is also a syndetic set.
Thus $(x,x,\dotsc,x)$ is an $\mathcal{F}_s$-transitive point of $(X^r,f^{(\mathbf{a})})$.
By Lemma~\ref{lem:minimal-Fs}, $(X^r,f^{(\mathbf{a})})$ is minimal.
\end{proof}

Let $\mathcal{F}_s[\infty]=\bigcap\limits_{i=1}^\infty \mathcal{F}_s[\mathbf{a}_i]$ for $i\in \mathbb{N}$,
where $\mathbf{a}_i=(1,2,\dotsc,i)$.
It is easy to check that a subset $F$ of $\mathbb{N}$ is in $\mathcal{F}_s[\infty]$ if and only if
for every $r\in\mathbb{N}$ and every $n_1, n_2, \dotsc, n_{r}\in \mathbb{Z}_+$, there exists a syndetic set $F'$
such that $F'+n_1,2F'+n_2,\dotsc,rF'+n_{r}\subset F$. We have the following characterization of multi-minimality.

\begin{thm}
Let $(X,f)$ be a dynamical system.
Then the following conditions are equivalent.
\begin{enumerate}
\item $(X,f)$ is multi-minimal;
\item  $(X,f)$ is $\mathcal{F}_s[\infty]$-point transitive;
\item $Trans_{\mathcal{F}_s[\infty]}(X,f)=X$.
\end{enumerate}
\end{thm}
\begin{proof}
(1) $\Rightarrow$ (3)
Assume that $(X,f)$ is multi-minimal. By Theorem~\ref{thm:mul-minimal-vector-Family},
we have
\[Trans_{\mathcal{F}_s[\mathbf{a}_i]}(X,f)=X,  \text{ for every } i\in\mathbb{N}.\]
Now the result follows from the fact that
\[Trans_{\mathcal{F}_s[\infty]}(X,f)=\bigcap\limits_{i=1}^\infty Trans_{\mathcal{F}_s[\mathbf{a}_i]}(X,f).\]

(3) $\Rightarrow$ (2) is obvious.

(2) $\Rightarrow$ (1)  Assume that $(X,f)$ is $\mathcal{F}_s[\infty]$-point transitive.
Let $x\in Trans_{\mathcal{F}_s[\infty]}(X,f)$ and $U$ be a non-empty open subset of $X$.
Then $N(x,U)\in\mathcal{F}_s[\infty]$, which implies that
$N(x,U)\in\mathcal{F}_s[\mathbf{a}_i]$ for every $i\in\mathbb{N}$.
Hence $(X,f)$ is $\mathcal{F}_s[\mathbf{a}_i]$-point transitive for every $i\in\mathbb{N}$.
By Theorem~\ref{thm:mul-minimal-vector-Family}, it is multi-minimal.
\end{proof}

\medskip
\noindent\emph{Acknowledgements.}
The first and third author were supported by National Nature Science Funds of China (11071084).
The second author was supported in part by STU Scientific Research Foundation for Talents (NTF12021) and
National Natural Science Foundation of China (11171320).
The authors would like to thank to the anonymous referee, whose remarks resulted in
substantial improvements to this paper, in particular for his contribution to Proposition~\ref{prop:delta-transitive-set}.

\end{document}